\documentclass[12pt]{amsart}

\usepackage{amsthm,amsmath,amsfonts}
\usepackage{enumerate}

\makeatletter
\@namedef{subjclassname@2010}{%
\textup{2010} Mathematics Subject Classification}
\makeatother

\allowdisplaybreaks

\theoremstyle{plain}
\newtheorem{thm}{Theorem}[section]
\newtheorem{defn}[thm]{Definition}
\newtheorem{lem}[thm]{Lemma}
\newtheorem{prop}[thm]{Proposition}
\newtheorem{cor}[thm]{Corollary}
\newtheorem{rem}[thm]{Remark}
\newtheorem{exam}[thm]{Example}

\title[Fourier analysis of vector measures on compact groups]{Fourier analysis associated to a vector measure on a compact group}
\author{Manoj Kumar}
\address{Department of Mathematics \endgraf Indian Institute of Technology Delhi \endgraf Delhi - 110016 \endgraf India}
\email{manojk9t3@gmail.com}
\author{N. Shravan Kumar}
\address{Department of Mathematics \endgraf Indian Institute of Technology Delhi \endgraf Delhi - 110016 \endgraf India}
\email{shravankumar@maths.iitd.ac.in}

\begin{document}

\begin{abstract}
In this paper, we introduce and study the Fourier transform of functions which are integrable with respect to a vector measure on a compact group (not necessarily abelian). We also study the Fourier transform of vector measures. We also introduce and study the convolution of functions from $L^p$-spaces associated to a vector measure. We prove some analogues of the classical Young's inequalities. Similarly, we also study convolution of a scalar measure and a vector measure.
\end{abstract}


\maketitle

\section{Introduction}
Let $G$ be a locally compact abelian group with a fixed Haar measure. Then the Fourier transform on $L^1(G)$ is very well known. Also, it is a known fact that this Fourier transform can be extended to $M(G),$ the space of all complex Radon measures on $G,$ called the Fourier-Stieltjes transform. Recently, J. M. Calabuig et al. \cite{CFNP}, have introduced and studied the Fourier transform of functions which are integrable with respect to a vector measure on a compact abelian group. This was extended by O. Blasco \cite{B} to the space of vector measures on compact abelian groups.

Let $G$ be a compact group. The main aim of this paper is to initiate a systematic study of the Fourier analysis of functions which are integrable with respect to a vector measure on $G.$ Let $\nu$ be a $\sigma$-additive vector measure. In section 4, we define the Fourier transform of functions in $L^1(\nu).$ In this section, we also discuss the Fourier transform of functions which are weakly integrable with respect to an absolutely continuous vector measure. We also show that these two notions of Fourier transform coincide if the function is integrable. 

It is a folklore that one will have to deal with matrices of higher orders, if one leaves the realm of abelian groups. Moreover, if the measure is a vector measure, then the entries of the matrix are from the underlying vector space. Thus, in order to make sense of the norm of these matrices, one is forced to assume that the underlying vector spaces are operator spaces, rather than just Banach spaces, unlike the case of compact abelian groups.  This crucial point is obviously at the root of the notion of the Fourier transform.

In order to define the Fourier transform of a function defined on a non-abelian compact group, which is integrable with respect to a vector measure, one needs to provide a meaning for the integration of a vector-valued function with respect to a vector measure. In 2001, G. F. Stef\'ansson \cite{S} developed the theory of integration of a vector-valued function with respect to a vector measure. Let $X$ and $Y$ be operator spaces which are also Banach spaces, $f$ a $Y$-valued function and $\nu$ a $X$-valued measure. In Section 3, we give an outline about the integration of $f$ with respect to $\nu.$ Further, in Theorem \ref{CBInt} an operator space structure on $L^1(\nu)$ is also provided.

In the classical Fourier analysis, an easy consequence of the definition of the Fourier transform of an $L^1$-function is that the Fourier transform is a bounded operator. In Theorems \ref{CBFTO} and \ref{CBFOWIF}, we show that the Fourier transform operators defined in Section 4 are completely bounded.

One of the classical results of the Fourier analysis on locally compact abelian groups is the Riemann-Lebesgue Lemma. Example \ref{EFRLL} of this paper gives an example of a vector measure on an infinite compact group where the analogue of the Riemann-Lebesgue Lemma fails.

Later, in Section 7, we define the Fourier transform of vector measures. Again operator spaces play a dominant role. In particular, we show that the Fourier transform operator on the space of vector measures is completely bounded. Finally, we also find a sufficient condition on the Banach space for an analogue of the Riemann-Lebesgue Lemma to hold.

In Section 6, we define the convolution of functions from $L^p$-spaces associated to a vector measure. In a similar spirit, in Section 8, we study the convolution of a scalar measure and a vector measure. We also find the Fourier transform of convolution. We prove some analogues of the Young's inequality corresponding to each convolution. Finally, in Section 9, we prove integrability properties of the convolution under various assumptions on the underlying $L^p$-spaces. We also consider the classical convolution between functions in $L^p$-spaces with respect to the Haar measure and functions in the $L^p$-spaces associated to a vector measure. This is done in Theorem \ref{T4}.

We begin with some of the required preliminaries in the next section.

\section{Preliminaries}
\subsection{Fourier analysis on compact groups} 
Let $G$ be a compact Hausdorff group and let $m_G$ denote the normalized positive Haar measure on $G.$ For $1\leq p\leq\infty$,  $L^p(G)$ will denote the usual $p^{\mbox{th}}$-Lebesgue space with respect to the measure $m_G.$ It is well known that an irreducible unitary representation of a compact group $G$ is always finite-dimensional. Let $\widehat{G}$ be the set of all unitary equivalence classes of irreducible unitary representations of $G$. The set $\widehat{G}$ is called the unitary dual of $G$ and $\widehat{G}$ is given the discrete topology. 

Let $\{(X_\alpha,\|.\|_\alpha)\}_{\alpha\in\wedge}$ be a collection of Banach spaces. We shall denote by $\ell^\infty\mbox{-}\underset{\alpha\in\wedge}{\oplus}X_\alpha,$ the Banach space $\left\{(x_\alpha)\in\underset{\alpha\in\wedge}{\Pi}X_\alpha:\underset{\alpha\in\wedge}{\sup}\|x_\alpha\|_\alpha<\infty\right\}$ equipped with the norm $\|(x_\alpha)\|_\infty:=\underset{\alpha\in\wedge}{\sup}\|x_\alpha\|_\alpha.$ Similarly, we shall also denote by $c_0\mbox{-}\underset{\alpha\in\wedge}{\oplus}X_\alpha,$ the space consisting of those vectors $(x_\alpha)$ from $\ell^\infty\mbox{-}\underset{\alpha\in\wedge}{\oplus}X_\alpha$ which goes to $0$ as $\alpha\rightarrow\infty.$ It is clear that $c_0\mbox{-}\underset{\alpha\in\wedge}{\oplus}X_\alpha$ is a closed subspace of $\ell^\infty\mbox{-}\underset{\alpha\in\wedge}{\oplus}X_\alpha.$

Let $\pi$ be an irreducible unitary representation of $G$ on the Hilbert space $\mathcal{H}_\pi$ of dimension $d_\pi$ and let $(e_i^\pi)_{1\leq i\leq d_\pi}$ be an ordered orthonormal basis for $\mathcal{H_\pi}.$ Then for $t\in G,\ \pi(t)$ will denote the $d_\pi\times d_\pi$ matrix whose $(i,j)^{\mbox{th}}$-entry is given by $\pi(t)_{ij}=\langle\pi(t)e_j^\pi,e_i^\pi\rangle.$ For $f\in L^1(G),$ the Fourier transform of $f,$ denoted $\widehat{f},$ is defined as $$\widehat{f}(\pi)=\frac{1}{d_\pi}\int_Gf(t)\pi(t)^* \,dm_G(t),~[\pi]\in\widehat{G}.$$ Then $\widehat{f}(\pi)$ is also a $d_\pi\times d_\pi$ matrix whose $(i,j)^{\mbox{th}}$-entry is given by  $$\widehat{f}(\pi)_{ij}=\frac{1}{d_\pi}\int_Gf(t)\overline{\pi(t)_{ji}}\,dm_G(t).$$

 Note that the Fourier transform operator $f\mapsto \widehat{f}$ maps $L^1(G)$ into $\ell^\infty\mbox{-}\underset{[\pi]\in\widehat{G}}{\oplus}M_{d_\pi}.$ This operator is injective and bounded. If  $f,g\in L^1(G)$ then $\widehat{f*g}(\pi)=d_\pi\widehat{g}(\pi)\widehat{f}(\pi),~[\pi]\in\widehat{G}.$
 
 \begin{thm}\label{FT} Let $f\in L^2(G)$. Then,
 \begin{enumerate}[(i)]
 \item {\normalfont(Plancheral Theorem).} $\|f\|_2^2=\sum_{[\pi]\in\widehat{G}}d_\pi^3\ tr(\widehat{f}(\pi)^*\widehat{f}(\pi)).$
 \item {\normalfont(Inversion Theorem).} $f=\sum_{[\pi]\in\widehat{G}}d_\pi^2\ tr(\widehat{f}(\pi)\pi(\cdot)),$ where the series converges in the $L^2(G)$-norm.
 \end{enumerate}
 \end{thm}
\noindent For more details on compact groups, we refer to \cite{F, HR}.

\subsection{Vector measure} 
Let $G$ be a compact group. Let $X$ be a complex Banach space and let $\nu$ be a $\sigma$-additive $X$-valued vector measure on $G.$ Let $X^\prime$ be the dual of $X$ and let $B_{X^\prime}$ be the closed unit ball in $X^\prime.$ For each $x^\prime\in X^\prime,$ we shall denote by $\langle\nu,x^\prime\rangle,$ the corresponding scalar valued measure for the vector measure $\nu,$ which is defined as $\langle\nu,x^\prime\rangle(A)=\langle\nu(A),x^\prime\rangle,\  A\in\mathfrak{B}(G).$ A set $A\in\mathfrak{B}(G)$ is said to be $\nu$-null if $\nu(B)=0$ for every $B\subset A.$ The variation of $\nu,$ denoted $|\nu|,$ is a positive measure defined as follows: For a set $A\in\mathfrak{B}(G),$ $$|\nu|(A)=\sup\left\{\sum_{E\in\rho}\|\nu(E)\|: \rho\mbox{ the finite partition of A}\right\}.$$ The vector measure $\nu$ is said to be measure of bounded variation if $|\nu|(G)<\infty.$ The semivariation of $\nu$ on a set $A\in\mathfrak{B}(G)$ is given by $\|\nu\|(A)=\underset{x^\prime\in B_{X^\prime}}\sup |\langle\nu,x^\prime\rangle|(A),$ where $|\langle\nu,x^\prime\rangle|$ is the total variation of the scalar measure $\langle\nu,x^\prime\rangle.$ Let $\|\nu\|$ denote the quantity $\|\nu\|(G).$ The vector measure $\nu$ is said to be {\it absolutely continuous with respect to a non-negative scalar measure $\mu$} if $\underset{\mu(A)\rightarrow 0}{\lim}\nu(A)=0,\ A\in\mathfrak{B}(G).$ The vector measure $\nu$ is said to be {\it regular} if for each $\epsilon>0$ and $A\in\mathfrak{B}(G)$ there exist an open set $U$ and a closed set $F$ with $F\subset A\subset U$ such that $\|\nu\|(U\setminus F)<\epsilon$. 

We shall denote by $M(G,X)$ the space of all $\sigma$-additive $X$-valued vector measures on $G.$ Further, we shall denote by $M_{ac}(G,X)$ the subspace consisting of $X$-valued vector measures which are absolutely continuous with respect to the Haar measure $m_G.$ A Banach space $X$ is said to have the {\it Radon-Nikodym Property} with respect to $(G,\mathfrak{B}(G),m_G)$ if for each measure $\nu\in M_{ac}(G,X)$ of bounded variation there exists $f\in L^1(G,X)$ such that $d\nu=f\,dm_G.$ We shall denote by $\mathcal{M}(G,X)$ the subspace of all $X$-valued regular vector measures. Note that $M_{ac}(G,X)\subset\mathcal{M}(G,X).$

A complex valued function $f$ on $G$ is said to be $\nu$-weakly integrable if $f\in L^1(|\langle\nu,x^\prime\rangle|),$ for all $x^\prime\in X^\prime.$ We shall denote by $L^1_w(\nu)$ the Banach space of all $\nu$-weakly integrable functions equipped with the norm $$\|f\|_\nu=\underset{x^\prime\in B_{X^\prime}}{\sup}\int_G |f|\,d|\langle\nu,x^\prime\rangle|.$$ A $\nu$-weakly integrable function $f$ is said to be $\nu$-integrable if for each $A\in\mathfrak{B}(G)$ there exists a unique $x_A\in X$ such that $\int_A f\,d\langle\nu,x^\prime\rangle=\langle x_A,x^\prime\rangle,\ x^\prime\in X^\prime.$ The vector $x_A$ is denoted by $\int_Af\,d\nu.$ We shall denote by $L^1(\nu)$ the space of all $\nu$-integrable functions and it is also a Banach space when equipped with the $\|\cdot\|_\nu$ norm. Now, for $1\leq p<\infty$ we say that $f\in L^p(\nu)$ (respectively $f\in L^p_w(\nu)$) if $f^p\in L^1(\nu)$ (respectively $f^p\in L^1_w(\nu)).$ The spaces $L^p(\nu)$ and $L^p_w(\nu)$ are Banach spaces when equipped with the norm $\|f\|_{\nu,p}=\|f^p\|^{1/p}_\nu.$ Let $L^\infty(\nu)=L^\infty_w(\nu)$ denote the space of all $\nu$-a.e. bounded functions. The space $S(G),$ consisting of all simple functions on $G$, is dense in $L^p(\nu),~1\leq p<\infty.$

The following result is proved in \cite{B} when $G$ is abelian. We are omitting the proof as the proof given in \cite[Lemma 2.1]{B} works for a more general case.
\begin{lem}\label{dense}
Let $\nu$ be a $X$-valued regular vector measure on $G.$ Then the space $C(G)$ of all continuous functions on $G$ is dense in $L^p(\nu),$ for $1\leq p<\infty.$
\end{lem}

If $f\in L^1(\nu)$ then $\nu_f(A)=\int_Af\,d\nu,$ for $A\in\mathfrak{B}(G),$ defines a $X$-valued measure on $G$ with $\|\nu_f\|=\|f\|_\nu$. If $f\in L^1_w(\nu)$ then $\nu_f$ is a $X^{\prime\prime}$-valued measure on $G$ given by $\langle\nu_f(A),x^\prime\rangle=\int_Af\,d\langle\nu,x^\prime\rangle, ~A\in\mathfrak{B}(G)$ and $x^\prime\in X^\prime.$

For $1\leq p<\infty,$ we shall denote by $\|\nu\|_{p,m_G},$ the $p$-semivariation of $\nu$ with respect to $m_G$, given by,  $$\|\nu\|_{p,m_G}=\sup\left\{\left\|\underset{A\in\rho}\sum\alpha_A\nu(A)\right\|_X:\rho\mbox{ the finite partition with }\underset{A\in\rho}\sum\alpha_A\chi_A\in B_{L^{p^\prime}(G)}\right\}$$ and for $p=\infty,$ $\|\nu\|_{\infty,m_G}=\underset{m_G(A)>0}\sup\frac{\|\nu(A)\|}{m_G(A)}.$ Let $M_p(G,X)$ denote the space of all $X$-valued vector measures with finite $p$-semivariation. We shall denote by $S(G,X)$ the space of all $X$-valued simple functions on $G.$ Further, for $1\leq p\leq\infty,$ we shall denote by $P_p(G,X)$ the closure of the space $S(G,X)$ in $M_p(G,X),$ where $P_p(G,X)$ is equipped with the norm $$\|\phi\|_{P_p(G,X)}=\|\nu_\phi\|_{p,m_G}=\sup_{x^\prime\in B_{X^\prime}}\|\langle\phi,x^\prime\rangle\|_p,\ \phi\in S(G,X).$$ Note that the space $C(G,X),$ consisting of all $X$-valued continuous functions on $G$, is dense in $P_p(G,X),~1\leq p<\infty$ and closed in $P_\infty(G,X).$

We denote by $T_\nu$ the operator from $C(G)$ to $X$ given by $T_\nu(f)=\int_Gf\ d\nu.$ If $\nu\in \mathcal{M}(G,X)$ then $T_\nu$ is a weakly compact operator and if $\nu\in M_p(G,X),\ 1<p\leq\infty,$ then $T_\nu$ can be extended to a bounded linear operator from $L^{p^\prime}(G)$ to $X$ with $\|T_\nu\|_{\mathcal{B}(L^{p^\prime}(G),X)}=\|\nu\|_{p,m_G}.$ For more details on vector measures and integration with respect to vector measures, we refer to \cite{D,DU, ORP}.

\subsection{Operator Spaces} 
In order to deal with compact groups in the non-abelian setting, one has to deal with matrices with vector-valued entries. Also, one has to be able to define norms of such matrices. Therefore, it is natural to deal only with the operator spaces, rather than with just Banach spaces. We shall now present some basics on operator spaces.

Let $X$ be a linear space. By $M_n(X)$ we shall mean the space of all $n\times n$ matrices with entries from the space $X.$ An {\it operator space} is a complex vector space $X$ together with an assignment of a norm $\|\cdot\|_n$ on the matrix space $M_n(X),$ for each $n\in\mathbb{N},$ such that
\begin{enumerate}[(i)]
\item $\|x\oplus y\|_{m+n}=max\{\|x\|_m,\|y\|_n\}$ and
\item $\|\alpha x\beta\|_n\leq\|\alpha\|\|x\|_m\|\beta\|$
\end{enumerate}
for all $x\in M_m(X),$ $y\in M_n(X),$ $\alpha\in M_{n,m}$ and $\beta\in M_{m,n}.$ It is clear from the definition that if $X$ is an operator space then $X^\prime$ is also an operator space where $M_n(X^\prime)$ is given the norm coming from the identification of $M_n(X^\prime)$ with $M_n(X)^\prime.$

It follows from the axiom (i) of the above definition that the inclusion from $M_r(X)$ into $M_{r+1}(X)$ is an isometry. It is also clear that if $x\in X$ and $\alpha\in M_n$ then $\|\alpha\otimes x\|_n=\|\alpha\|\|x\|.$ Hence it follows that, if $[x_{ij}]\in M_n(X)$ then,  $$\underset{i,j}{\max}\|x_{ij}\|\leq\|[x_{ij}]\|_n\leq n\ \underset{i,j}{\max}\|x_{ij}\|.$$ Thus $X$ is complete if and only if $M_n(X)$ is complete for some $n\in\mathbb{N}$ if and only if $M_n(X)$ is complete for all $n\in\mathbb{N}.$ 

Let $X$ and $Y$ be operator spaces and let $\varphi:X\rightarrow Y$ be a linear transformation. For any $n\in\mathbb{N},$ the {\it $n^{\mbox{th}}$-amplification} of $\varphi,$ denoted $\varphi_n,$ is defined as a linear transformation $\varphi_n:M_n(X)\rightarrow M_n(Y)$ given by $\varphi_n([x_{ij}]):=[\varphi(x_{ij})].$ The linear transformation $\varphi$  is said to be {\it completely bounded} if $\sup\{\|\varphi_n\||n\in\mathbb{N}\}<\infty.$ We shall denote by $\mathcal{CB}(X,Y)$ the space of all completely bounded linear mappings from $X$ to $Y$ equipped with the norm, denoted $\|\cdot\|_{cb},$ $$\|\varphi\|_{cb}:=\sup\{\|\varphi_n\||n\in\mathbb{N}\},\ \varphi\in\mathcal{CB}(X,Y).$$We shall say that $\varphi$ is a {\it complete isometry} if $\varphi_n$ is an isometry $\forall\ n\in\mathbb{N}.$

We would like to remark here that by Ruan's theorem, on the characterization of abstract operator spaces, there exists a Hilbert space $\mathcal{H}$ and a closed subspace $Y\subseteq\mathcal{B}(\mathcal{H})$ such that $X$ and $Y$ are completely isometric. 

Given two operator spaces $X_1\subseteq\mathcal{B}(\mathcal{H}_1)$ and $X_2\subseteq\mathcal{B}(\mathcal{H}_2),$ we define their minimal tensor product, denoted $X_1\otimes_{min}X_2,$ as the completion of their algebraic tensor product $X_1\otimes X_2$ inside $\mathcal{B}(\mathcal{H}_1\otimes_2\mathcal{H}_2),$ where $\otimes_2$ denotes the Hilbert space tensor product. It is worth noting that, if $X$ is an operator space then $M_n\otimes_{min}X$ and $M_n(X)$ are completely isometric.

If $[x_{ij}]\in M_{n}(X)$ and $[x^\prime_{kl}]\in M_m(X^\prime),$ then the matrix pairing between $[x_{ij}]$ and $[x^\prime_{kl}]$ is the $mn\times mn$ matrix given by $$\langle\langle[x_{ij}],[x^\prime_{kl}]\rangle\rangle:=[\langle x_{ij},x^\prime_{kl}\rangle].$$

For any undefined notions or for further reading on operator spaces, the reader is asked to refer \cite{ER} or \cite{P}. 

Throughout  this paper, $G$ will always denote a compact Hausdorff group, $X$ an operator space which is also a Banach sapce and $\nu$ a $\sigma$-additive $X$-valued vector measure.

\section{Tensor integrability}
Let $X$ and $Y$ be two operator spaces which are also Banach spaces. In this section, we give an outline about the integration of $Y$-valued functions with respect to the $X$-valued vector measure $\nu.$ For the proof of the results, we refer to \cite{S}. Finally, we show that the spaces $L^1(\nu),$ $L^1_w(\nu)$ and $M(G,X)$ are operator spaces.
\begin{defn}
A function $f:G\rightarrow Y$ is said to be $\nu$-measurable if there exists a sequence $(\phi_n)$ of $Y$-valued simple functions on $G$ such that $\underset{n}{\lim}\|\phi_n(t)-f(t)\|_Y=0$ $\nu$-a.e..
\end{defn}
For a $\nu$-measurable function $f:G\rightarrow Y,$ we shall denote by $N(f)$ the quantity $$\underset{x^\prime\in B_{X^\prime}}{\sup}\int_G\|f\|\,d|\langle\nu,x^\prime\rangle|.$$
\begin{defn}
A $\nu$-measurable $Y$-valued function $f$ is said to be $\otimes_{min}$-integrable if there exists a sequence $(\phi_n)$ of $Y$-valued simple functions such that $\underset{n}{\lim}\ N(\phi_n-f)=0.$
\end{defn}
Note that, for each $A\in\mathfrak{B}(G),$ the sequence $\left(\int_A\phi_n\,d\nu\right)$ is a Cauchy sequence in $Y\otimes_{min}X$ and hence converges to a limit in $Y\otimes_{min}X.$ We shall denote the limit by $\int_Af\,d\nu\in Y\otimes_{min}X,$ called as $\otimes_{min}$-integral of $f$ over $A$ with respect to $\nu.$ Let $L^1(\nu,Y,X)$ denote the space of all such $\otimes_{min}$-integrable functions on $G$. The space $L^1(\nu,Y,X)$ becomes a Banach space when it is equipped with the $N(\cdot)$ norm.
\begin{thm}
    Let $f$ be a $\nu$-measurable function. Then $f\in L^1(\nu,Y,X)$ if and only if $\|f\|\in L^1(\nu).$
\end{thm}
\begin{cor}Let $f$ be a $\nu$-measurable function.
\begin{enumerate}[(i)]
\item If $f$ is bounded then it is $\otimes_{min}$-integrable.
\item If $\|f\|\leq\|g\|$ $\nu$-a.e. for some $g\in L^1(\nu,Y,X),$ then  $f\in L^1(\nu,Y,X)$.
\end{enumerate}
\end{cor}
\begin{prop}\label{evaluation}
If $f\in L^1(\nu,Y,X)$, then for $y^\prime\in Y^\prime$ and $T\in \mathcal{CB}(X),$ $$(y^\prime\otimes T)\left(\int_Af\,d\nu\right)=\int_A\langle f ,y^\prime\rangle \,d(T\circ\nu),~A\in\mathfrak{B}(G).$$ 
\end{prop}
Our next aim is to provide an operator space structure for $L^1(\nu),$ $L^1_w(\nu)$ and $M(G,X).$
\begin{thm}\label{CBInt}\mbox{}
\begin{enumerate}[(i)]
\item The spaces $M_n(L^1(\nu))$ and $L^1(\nu,M_n,X)$ are isomorphic via the mapping $[f_{ij}]\mapsto\widetilde{f},$ where $\widetilde{f}(\cdot)=[f_{ij}(\cdot)].$
\item The space $L^1(\nu)$ is an operator space with respect to the matrix norm arising from the identification given in (i). 
\item The mapping $f\mapsto\underset{G}{\int}fd\nu$ from $L^1(\nu)$ into $X$ is completely bounded.
\end{enumerate}
\end{thm}
\begin{proof}
The proof of (i) and (ii) are routine checks. The proof of (iii) follows from  the proof of \cite[Corollary 3]{CSb}.
\end{proof}
\begin{defn}
	A function $f:G\rightarrow Y$ is said to be generalized weak $\otimes$-integrable with respect to $\nu$ if $\langle f,y^\prime\rangle\in L^1_w(\nu),~\forall\ y^\prime\in Y^\prime.$ 
\end{defn}
we shall denote by $gen\mbox{-}L^1_w(\nu,Y,X)$ the space consisting of $Y$-valued generalized weak $\otimes$-integrable functions on $G.$ It is a normed linear space when equipped with the norm given by $$N_{w}(f)=\sup_{y^\prime\in B_{Y^\prime}}\|\langle f,y^\prime\rangle\|_\nu.$$
The proofs of the following theorems are routine checks, so we shall omit them.
\begin{thm}\label{CBInt2}\mbox{}
	\begin{enumerate}[(i)]
		\item The spaces $M_n(L^1_w(\nu))$ and $\mbox{gen-}L^1_w(\nu,M_n,X)$ are isomorphic via the mapping $[f_{ij}]\mapsto\widetilde{f},$ where $\widetilde{f}(\cdot)=[f_{ij}(\cdot)].$
		\item The space $L^1_w(\nu)$ is an operator space with respect to the matrix norm arising from the identification given in (i).
	\end{enumerate}
\end{thm}
\begin{thm}\label{measure}\mbox{}
	\begin{enumerate}[(i)]
		\item The space $M_n(M(G,X))$ and $M(G,M_n(X))$ are isomorphic via the mapping $[\nu_{ij}]\mapsto\tilde{\nu},$ where $\tilde{\nu}(A)=[\nu_{ij}(A)],\ A\in\mathfrak{B}(G).$
		\item The space $M(G,X)$ is an operator space with respect to the matrix norm arising from the identification given in (i).
	\end{enumerate}
\end{thm}

\section{Fourier transform for $L^1(\nu)$ and $L^1_w(\nu)$}
In this section, we define the notion of Fourier transform of functions in $L^1(\nu)$ and $L^1_w(\nu).$ We show that the Fourier transform operator is completely bounded. We also provide an example where the analogue of the Riemann-Lebesgue Lemma fails. Finally, we provide a subclass of vector measures $\nu$ and a subclass of functions from $L^1_w(\nu)$ for which the classical Plancheral identity holds.

We first define the notion of Fourier transform of functions in $L^1(\nu)$ using the fact that, if $f\in L^1(\nu)$ and $[\pi]\in\widehat{G},$ then  $f(\cdot)\pi(\cdot)^\ast\in L^1(\nu,\mathcal{B}(\mathcal{H}_\pi),X).$ 
\begin{defn}\label{D2}
The Fourier transform of a function $f\in L^1(\nu)$ at $[\pi]\in\widehat{G},$ with respect to the vector measure $\nu,$ is defined by  $$\widehat{f}^\nu(\pi)=\frac{1}{d_\pi}\int_Gf(t)\pi(t)^*\,d\nu(t)\in \mathcal{B}(\mathcal{H}_\pi)\otimes_{min}X.$$
\end{defn}
\begin{rem}\label{RforFTM}
Since the representation $\pi$ is of $d_\pi$-dimension, by fixing an ordered orthonormal basis for $\mathcal{H}_\pi$, the space $\mathcal{B}(\mathcal{H}_\pi)$ can be identified with $M_{d_\pi}.$ Thus the Fourier transform of $f\in L^1(\nu)$ at $[\pi]\in\widehat{G}$ belongs to $M_{d_\pi}(X).$ Furthermore, the entries of the matrix $\widehat{f}^\nu(\pi)$ are given by the following elements of $X.$ For $1\leq i,j\leq d_\pi,$ let $P_{ij}^\pi$ denote the mapping from $M_{d_\pi}$ into $\mathbb{C},$ maps a $d_\pi\times d_\pi$ matrix to its $(i,j)^{\mbox{th}}$-entry. Thus, by Proposition \ref{evaluation}, $$\widehat{f}^\nu(\pi)_{ij}=\left(P_{ij}^\pi\otimes Id_{X}\right)\left(\widehat{f}^\nu(\pi)\right)=\frac{1}{d_\pi}\int_Gf(t)(\pi(t)^*)_{ij}\,d\nu(t)=\frac{1}{d_\pi}\int_Gf(t)\overline{\pi(t)_{ji}}\,d\nu(t)$$
\end{rem}
\begin{exam}\label{FTExam}
Let $1\leq p<\infty$ and let $T:L^p(G)\rightarrow X$ be any completely bounded operator. Consider a vector measure $\nu$ associated with $T$ given by $\nu(A)=T(\chi_A),\ A\in\mathfrak{B}(G).$ Note that, by \cite[Proposition 4.4]{ORP}, $L^p(G)\subset L^1(\nu).$ Also, for any $A\in \mathfrak{B}(G)$ and $f\in L^p(G),$ we have, $$\int_A f \,d\nu=T(f\chi_A).$$ Then, for $f\in L^p(G),$ we have, $$\widehat{f}^\nu(\pi)=T_{d_\pi}\left(f\pi(\cdot)^\ast\right).$$
\end{exam}
We now show that the Fourier transform is a completely bounded operator.
\begin{thm}\label{CBFTO}\mbox{ }
\begin{enumerate}[(i)]
\item If $f\in L^1(\nu)$ then $\widehat{f}^\nu\in\ell^\infty\mbox{-}\underset{[\pi]\in\widehat{G}}{\oplus}M_{d_\pi}(X).$ In fact, $\underset{[\pi]\in\widehat{G}}{\sup}\|\widehat{f}^\nu(\pi)\|_{d_\pi}\leq\|f\|_\nu.$
\item The Fourier transform operator $\mathcal{F}^\nu$ from $L^1(\nu)$ to $\ell^\infty\mbox{-}\underset{[\pi]\in\widehat{G}}{\oplus}M_{d_\pi}(X)$ given by $\mathcal{F}^\nu(f)=\widehat{f}^\nu,$ is completely bounded.
\end{enumerate}
\end{thm}
\begin{proof}
Let $f\in L^1(\nu).$ We know that, for any $[\pi]\in\widehat{G},$ $|\pi(t)_{ij}|\leq 1\ \forall\ 1\leq i,j\leq d_\pi.$ Hence, $$\|\widehat{f}^\nu(\pi)_{ij}\|=\sup_{x^\prime\in B_{X^\prime}}|\langle\widehat{f}^\nu(\pi)_{ij},x^\prime\rangle|=\sup_{x^\prime\in B_{X^\prime}}\left|\frac{1}{d_\pi}\int_Gf(t)\overline{\pi(t)_{ji}}\,d\langle\nu,x^\prime\rangle(t)\right|\leq\frac{1}{d_\pi}\|f\|_\nu.$$ 
Thus, \begin{align*}
\sup_{[\pi]\in\widehat{G}}\|\widehat{f}^\nu(\pi)\|_{d_\pi}&\leq\sup_{[\pi]\in\widehat{G}}d_\pi\max_{1\leq i,j\leq d_\pi}\|\widehat{f}^\nu(\pi)_{ij}\|\leq\|f\|_\nu.
\end{align*}
Thus (i) follows. The proof of (ii) follows from Theorem \ref{CBInt}.
\end{proof}
A natural question that arises at this point is the validity of the Riemann Lebesgue Lemma for the Fourier transform, i.e., 
\begin{center}
{\it does  $\widehat{f}^\nu\in c_0\mbox{-}\underset{[\pi]\in\widehat{G}}{\oplus}M_{d_\pi}(X)$ whenever $f\in L^1(\nu)?$}
\end{center} 
The answer to the above mentioned question is negative in general, even negative for compact abelian groups. An example is provided here.
\begin{exam}\label{EFRLL}
Let $G$ be an infinite compact group. Then, consider the measure defined in Example \ref{FTExam} with $X=L^1(G)$ and $T$ the identity operator on $L^1(G).$ Let $0\neq f\in L^1(G)$ and $[\pi]\in\widehat{G}.$ Then, $\widehat{f}^\nu(\pi)=f\pi(\cdot)^\ast$ and therefore, $$\|\widehat{f}^\nu(\pi)\|_{M_{d_\pi}(L^1(G))}=\|f\pi(\cdot)^\ast\|_{M_{d_\pi}(L^1(G))}=\|f\|_{L^1(G)}.$$ Thus $\widehat{f}^\nu\notin c_0\mbox{-}\underset{[\pi]\in\widehat{G}}{\oplus}M_{d_\pi}(L^1(G)).$
\end{exam}
Now we study the Fourier transform of weakly integrable functions with respect to the vector measure $\nu$ under the assumption that $\nu\in M_{ac}(G,X).$ Since $\nu\in M_{ac}(G,X),$ it follows that $\langle\nu,x^\prime\rangle\in M_{ac}(G,\mathbb{C}),\ \forall\  x^\prime\in X^\prime.$ Thus, for a fixed $x^\prime\in X^\prime,$ by Radon-Nikodym theorem, there exists $h_{x^\prime}\in L^1(G)$ such that $d\langle\nu,x^\prime\rangle=h_{x^\prime}\,dm_G$ and hence for $f\in L^1_w(\nu)$, $$\int_G|f|(t)d|\langle\nu,x^\prime\rangle|(t)=\int_G|fh_{x^\prime}|(t)\,dm_G(t).$$ Thus, we have $fh_{x^\prime}\in L^1(G)$ for every $f\in L^1_w(\nu).$ With this as motivation, we define the Fourier transform of weakly integrable functions.
\begin{defn}\label{wFT}
Let $\nu\in M_{ac}(G,X).$ Then the Fourier transform of a function $f\in L^1_w(\nu)$ with respect to the vector measure $\nu$ is defined by $$\widehat{f}_\nu(x^\prime)(\pi)=\widehat{fh_{x^\prime}}(\pi),~[\pi]\in\widehat{G}\mbox{ and }x^\prime\in X^\prime.$$
\end{defn}
Since $fh_{x^\prime}\in L^1(G),$ it follows that $\widehat{fh_{x^\prime}}(\pi)\in M_{d_\pi}.$ Hence, for each $x^\prime\in X^\prime,$ we have  $\widehat{f}_\nu(x^\prime)\in \ell^\infty\mbox{-}\underset{[\pi]\in\widehat{G}}{\oplus}M_{d_\pi}.$

Our next theorem is an analogue of Theorem \ref{CBFTO}.
\begin{thm}\label{CBFOWIF}
Let $\nu\in M_{ac}(G,X).$
\begin{enumerate}[(i)]
\item If $f\in L^1_w(\nu)$ then $\widehat{f}_\nu\in\mathcal{B}\left(X^\prime,\ell^\infty\mbox{-}\underset{[\pi]\in\widehat{G}}{\oplus}M_{d_\pi}\right).$ In fact, $\underset{[\pi]\in\widehat{G}}{\sup}\|\widehat{f}_\nu(x^\prime)(\pi)\|_{M_{d_\pi}}\leq\|f\|_\nu\|x^\prime\|,\ x^\prime\in X^\prime.$
\item The Fourier transform operator $\mathcal{F}_\nu:L^1_w(\nu)\rightarrow \mathcal{B}\left(X^\prime, \ell^\infty\mbox{-}\underset{[\pi]\in\widehat{G}}{\oplus}M_{d_\pi}\right)$ given by $\mathcal{F}_\nu(f)=\widehat{f}_\nu$, is completely bounded.
\end{enumerate}
\end{thm}
\begin{proof}
	For $f\in L^1_w(\nu)$ and $x^\prime\in X^\prime,$ we have $fh_{x^\prime}\in L^1(G).$ Using the fact that if $g\in L^1(G)$ then $\underset{[\pi]\in\widehat{G}}{\sup} \|\widehat{g}(\pi)\|_{M_{d_\pi}} \leq \|g\|_1,$ we have,
	$$\sup_{[\pi]\in\widehat{G}}\|\widehat{f}_\nu(x^\prime)(\pi)\|_{M_{d_\pi}}=\sup_{[\pi]\in\widehat{G}}\|\widehat{fh_{x^\prime}}(\pi)\|_{M_{d_\pi}}\leq\|fh_{x^\prime}\|_1\leq\|f\|_\nu\|x^\prime\|.$$
    Thus (i) follows. The proof of (ii) follows as the proof of Theorem \ref{CBInt}(iii), once we use Theorem \ref{CBInt2}.
\end{proof}
\begin{rem}\label{E3}
If $\nu\in M_{ac}(G,X)$ and $f\in L^1(\nu),$ then for any $x^\prime\in X^\prime$ and $[\pi]\in\widehat{G},$
$$\langle\langle \widehat{f}^\nu(\pi),x^\prime\rangle\rangle=[\langle \widehat{f}^\nu(\pi)_{ij},x^\prime\rangle]=[\widehat{fh_{x^\prime}}(\pi)_{ij}]=\widehat{fh_{x^\prime}}(\pi)=\widehat{f}_\nu(x^\prime)(\pi).$$
\end{rem}
We now show the injectivity of the Fourier transform operator.
\begin{thm}[Uniqueness theorem]\label{T5} Let $\nu\in M_{ac}(G,X)$ and $f\in L^1_w(\nu)$. If $\widehat{f}_\nu=0$ then $f=0$ $\nu$-a.e..
\end{thm}
\begin{proof}
Let $f\in L^1_w(\nu)$ and let $x^\prime\in X^\prime.$ As $\widehat{f}_\nu=0$ it follows that $\widehat{f}_\nu(x^\prime)=0,$ i.e., $\widehat{fh_{x^\prime}}=0.$ Then by the classical Uniqueness theorem for Fourier transform we have that $fh_{x^\prime}=0$ $m_G$-a.e.. Hence, there exists $A\in\mathfrak{B}(G)$ such that $fh_{x^\prime}=0$ on $G\setminus A$ with $m_G(A)=0.$ Since $\nu\in M_{ac}(G,X)$ we get $|\langle\nu,x^\prime\rangle|(A)=0.$ So $A$ is $\nu$-null. Let $\widetilde{G\setminus A}=\{t\in G\setminus A:h_{x^\prime}(t)=0\}.$ Thus $f=0$ on $(G\setminus A)\setminus\widetilde{G\setminus A}.$ Now, note that $$|\langle\nu,x^\prime\rangle|(\widetilde{G\setminus A})=\int_{\widetilde{G\setminus A}}|h_{x^\prime}|\,dm_G=0,$$ which implies that $\widetilde{G\setminus A}$ is also $\nu$-null. Therefore, $$\int_G|f|\,d|\langle\nu, x^\prime\rangle|=\int_A|f|\,d|\langle\nu, x^\prime\rangle|+\int_{\widetilde{G\setminus A}}|f|\,d|\langle\nu, x^\prime\rangle|+\int_{(G\setminus A)\setminus\widetilde{G\setminus A}}|f|\,d|\langle\nu, x^\prime\rangle|=0.$$ Hence $f=0$ $\nu$-a.e..
\end{proof}
Using Remark \ref{E3} and Theorem \ref{T5} we have the following corollary.
\begin{cor}Let $\nu\in M_{ac}(G,X)$ and $f\in L^1(\nu).$ If $\widehat{f}^\nu=0$ then $f=0$ $\nu$-a.e..
\end{cor}
\begin{defn}Let $k\in[0,\infty).$ A vector measure $\nu$ is said to be $k$-scalarly bounded by $m_G$ if for any $x^\prime\in X^\prime$ and $A\in\mathfrak{B}(G),$ we have $|\langle\nu,x^\prime\rangle|(A)\leq km_G(A).$
\end{defn}
\begin{lem}\mbox{}\label{Mp}
	\begin{enumerate}[(i)]
		\item If $\nu$ is $k$-scalarly bounded by $m_G,$ then $\nu\in M_{ac}(G,X)$ and for each $x^\prime\in X^\prime$ there exists $h_{x^\prime}\in L^p(G)$ such that $d\langle\nu,x^\prime\rangle=h_{x^\prime}\,dm_G$ for every $1\leq p\leq\infty.$
		\item If $\nu\in M_p(G,X),\ 1<p\leq\infty,$ then $\nu\in M_{ac}(G,X)$ and for each $x^\prime\in X^\prime$ there exists $h_{x^\prime}\in L^p(G)$ such that $d\langle\nu,x^\prime\rangle=h_{x^\prime}\,dm_G.$ Moreover,  $L^{p^\prime}(G)\subset L^1(\nu),$ where $p^\prime$ is the conjugate exponent of $p.$
	\end{enumerate}
\end{lem}
\begin{proof}
	(i) is clear from the definition. We now prove (ii). Let $\nu\in M_p(G,X),\ 1<p\leq\infty.$ By \cite[Pg. 248 and Pg. 259]{D}, it follows that $\nu\in M_{ac}(G,X)$ and the operator $T_\nu$ extends to a bounded linear operator from $L^{p^\prime}(G)$ to $X$ with $\|T_\nu\|=\|\nu\|_{p,m_G}.$ Thus, $T_\nu^*(X^\prime)\subset L^p(G).$ Now, let $x^\prime\in X^\prime.$ Then $T_\nu^*(x^\prime)=\langle\nu,x^\prime\rangle$ and hence there exists $h_{x^\prime}\in L^p(G)$ such that $d\langle\nu,x^\prime\rangle=h_{x^\prime}\,dm_G.$ If $f\in C(G)$ then $$\left\|\int_Gf\,d\nu\right\|=\|T_\nu(f)\|\leq\|\nu\|_{p,m_G}\|f\|_{p^\prime}.$$ Hence the conclusion that $L^{p^\prime}(G)\subset L^1(\nu)$ follows from the density of $C(G)$ in $L^{p^\prime}(G).$
\end{proof}
In the next theorem, we provide a subclass of $L^1_w(\nu)$ and vector measures for which the Fourier transform satisfies the classical Plancheral Theorem and the Inversion Theorem.
\begin{thm}
	Let either $\nu\in M_4(G,X)$ or $\nu$ be $k$-scalarly bounded by $m_G$. If $f\in L^1_w(\nu)\cap L^4(G),$ then for $x^\prime\in X^\prime,$ 	\begin{enumerate}[(i)]
		\item {\normalfont(Plancheral Theorem).} $$\|fh_{x^\prime}\|_2^2=\sum_{[\pi]\in\widehat{G}}d_\pi^3\  tr((\widehat{f}_\nu(x^\prime)(\pi))^*\widehat{f}_\nu(x^\prime)(\pi)).$$
		\item {\normalfont(Inversion Theorem).} $$fh_{x^\prime}=\sum_{[\pi]\in\widehat{G}}d_\pi^2\ tr(\widehat{f}_\nu(x^\prime)(\pi)\pi(\cdot)),$$ where the above series converges in the $L^2(G)$-norm.
	\end{enumerate}
\end{thm}
\begin{proof}
	Let either $\nu\in M_4(G,X)$ or $\nu$ is $k$-scalarly bounded by $m_G.$ Let $f\in L^1_w(\nu)\cap L^4(G)$ and $x^\prime\in X^\prime.$  Then, by Lemma \ref{Mp}, we have $h_{x^\prime}\in L^4(G).$ Thus, by the H\"older's inequality, we have $fh_{x^\prime}\in L^2(G).$ Hence by Theorem \ref{FT}, the result follows.
\end{proof}

\section{Invariant measures}
In this section, we introduce the notion of invariant vector measures with respect to a homeomorphism. The results of this section will be used later.

Let $h:G\rightarrow G$ be a homeomorphism, for example translation ($\tau_t(s)=st$) or inversion ($i(t)=t^{-1}$). For a measurable function  $f:G\rightarrow\mathbb{C},$ the function $f_h:G\rightarrow\mathbb{C}$ given by $f_h=f\circ h^{-1}$, is also a measurable function. For example $\tau_tf(s):=f_{\tau_t}(s)=f(st^{-1})$ and $\tilde{f}(t):=f_i(t)=f(t^{-1}).$ Define $\nu_h(A)=\nu(h(A)),~A\in\mathfrak{B}(G).$ Then $\nu_h$ is also a vector measure.

Now we define the notion of invariance of a vector measure.
\begin{defn}A vector measure $\nu$ is said to be semivariation $h$-invariant if $\|(\nu_h)_\phi\|=\|\nu_\phi\|,~\forall~\phi\in S(G).$
\end{defn}
\begin{prop}\label{A4}
A vector measure $\nu$ is semivariation $h$-invariant if and only if $L^1(\nu)=L^1(\nu_h)$ isometrically.
\end{prop}
\begin{proof}
	By density, it is enough to prove for simple functions. Let $\phi\in S(G).$ Now let the measure $\nu$ be semivariation $h$-invariant. Then $\|\phi\|_{\nu_h}=\|(\nu_h)_\phi\|=\|\nu_\phi\|=\|\phi\|_\nu.$ Conversely, since $L^1(\nu)=L^1(\nu_h)$ isometrically, it follows that $\|(\nu_h)_\phi\|=\|\phi\|_{\nu_h}=\|\phi\|_\nu=\|\nu_\phi\|.$
\end{proof}
\begin{rem}\label{r2}If $\nu$ is semivariation $h$-invariant, then for $1\leq p<\infty,$ $$\|\phi\|_{\nu,p}=\|\phi^p\|_\nu^{1/p}=\|\nu_{\phi^p}\|^{1/p}=\|(\nu_{h})_{\phi^p}\|^{1/p}=\|\phi^p\|_{\nu_h}^{1/p}=\|\phi\|_{\nu_h,p},~\phi\in S(G).$$ By density of simple functions we have $L^p(\nu)=L^p(\nu_h)$ isometrically for $1\leq p<\infty$.
\end{rem}
\begin{defn}
A Banach function space $Z$ is said to be norm $h$-invariant if for each $f\in Z$ we have $f_h\in Z$ and $\|f_h\|_Z=\|f\|_Z.$
\end{defn}
\begin{prop}\label{R1}
	Let $\nu$ be a semivariation $h$-invariant vector measure and $1\leq p<\infty$. Then $L^p(\nu)$ is norm $h$-invariant.
\end{prop}
\begin{proof}
	By density, it is enough to prove for simple functions. Let $\phi\in S(G).$ Note that $$\|\phi\|_{\nu_h}=\|(\nu_h)_\phi\|=\|\nu_{\phi_h}\|=\|\phi_h\|_\nu.$$ Then, by Remark $\ref{r2}$ we have,
	\begin{align*}
	\|\phi_h\|_{\nu,p}=&\|(\phi_h)^p\|_\nu^{1/p}=\|(\phi^p)_h\|_\nu^{1/p}=\|\phi^p\|_{\nu_h}^{1/p}=\|\phi\|_{\nu_h,p}=\|\phi\|_{\nu,p}.\qedhere
	\end{align*}
\end{proof}
The proof of the next theorem is analogous to \cite[Theorem 5.10]{B} and hence we omit it.
\begin{thm}\label{A3}Let $1\leq p<\infty$ and let $\nu\in\mathcal{M}(G,X)$ be a semivariation translation invariant vector measure with $\nu(G)\neq 0$. Then $L^p(\nu)\subset L^p(G).$ Further $\|f\|_p\leq\|f\|_{\nu,p}\|\nu(G)\|^{-1/p}.$
\end{thm}

Our next result is an analogue of \cite[Proposition 2.41]{F}. As the proof of this proposition is just a routine check, we shall omit the proof.

\begin{prop}\label{A5}Let $\nu\in\mathcal{M}(G,X)$ be a semivariation translation invariant vector measure. If $f\in L^p(\nu),\ 1\leq p<\infty,$ then the mapping $s\mapsto\tau_sf$ is uniformly continuous from $G$ into $L^p(\nu).$
\end{prop}

\section{Convolution of functions associated to vector measures}
In this section, we define the notion of convolution of functions arising from $L^p$-spaces with respect to a vector measure. Our aim is to prove an analogue of the Young's inequality.

In Section 4, for $\nu\in M_{ac}(G,X),$ we observed that if $g\in L^1_w(\nu)$ then for each $x^\prime\in X^\prime$ there exists $h_{x^\prime}\in L^1(G)$ such that $gh_{x^\prime}\in L^1(G).$ With this observation in mind, we now define the convolution of two functions.
\begin{defn}\label{d1}
	Let $1\leq p\leq\infty.$ The convolution of the functions $f\in L^p(G)$ and $g\in L^1_w(\nu)$ with respect to the vector measure $\nu\in M_{ac}(G,X)$ is defined by $$f*_\nu g(x^\prime)=f*(gh_{x^\prime}),~x^\prime\in X^\prime.$$ 
\end{defn}
\begin{lem}\label{bddconvo}
	Let $1\leq p\leq\infty.$ If $f\in L^p(G)$ and $g\in L^1_w(\nu),$ then $f*_\nu g\in\mathcal{B}(X^\prime, L^p(G))$ with $\|f*_\nu g\|_{\mathcal{B}(X^\prime,L^p(G))}\leq\|f\|_p\|g\|_\nu.$ 
\end{lem}
\begin{proof}
	Let $x^\prime\in B_{X^\prime}.$ Then we have $\|f*_\nu g(x^\prime)\|_p=\|f*(gh_{x^\prime})\|_p\leq\|f\|_p\|gh_{x^\prime}\|_1\leq\|f\|_p\|g\|_\nu.$
\end{proof}
\begin{rem}\label{R2}
	Let $1\leq p<\infty.$ Observe that, if $\nu\in M_{ac}(G,X)$ is a semivariation translation invariant vector measure with $\nu(G)\neq 0$ then, by Theorem \ref{A3}, Definition \ref{d1} makes sense even if $f\in L^p(\nu)$ and $g\in L^1_w(\nu).$ In fact, we have $f*_\nu g\in\mathcal{B}(X^\prime, L^p(G))$ with $$\|f*_\nu g\|_{\mathcal{B}(X^\prime,L^p(G))}\leq\|f\|_{\nu,p}\|g\|_\nu\|\nu(G)\|^{-1/p}.$$
\end{rem}
The following theorem improves the above remark.
\begin{thm}\label{T1}
	Let $1\leq p<\infty$ and let $\nu\in M_{ac}(G,X)$ be a semivariation translation invariant vector measure with $\nu(G)\neq 0$. If $f\in L^p(\nu)$ and $g\in L^1_w(\nu),$ then $f*_\nu g\in \mathcal{B}(X^\prime,L^p_w(\nu))$ with $\|f*_\nu g\|_{\mathcal{B}(X^\prime,L^p_w(\nu))}\leq\|f\|_{\nu,p}\|g\|_\nu.$ In particular, if $g\in L^1(\nu),$ then $f*_\nu g\in \mathcal{B}(X^\prime,L^p(\nu)).$
\end{thm}
\begin{proof}
	Let $x^\prime, y^\prime\in B_{X^\prime}$ and $f\in L^p(\nu).$ For $g\in L^1_w(\nu),$ by the Minkowski's Integral Inequality and Proposition \ref{R1} we have, 
	\begin{align*}
	\|f*_\nu g(x^\prime)\|_{L^p(|\langle\nu,y^\prime\rangle|)}=&\|f*(gh_{x^\prime})\|_{L^p(|\langle\nu,y^\prime\rangle|)}=\left\|\int_G\tau_sfg(s)\,d\langle\nu,x^\prime\rangle(s)\right\|_{L^p(|\langle\nu,y^\prime\rangle|)}\\\leq&\int_G\|\tau_sf\|_{L^p(|\langle\nu,y^\prime\rangle|)}|g(s)|\,d|\langle\nu,x^\prime\rangle|(s)\\\leq&\int_G\|\tau_sf\|_{\nu,p}|g(s)|\,d|\langle\nu,x^\prime\rangle|(s)\leq\|f\|_{\nu,p}\|g\|_\nu.
	\end{align*} 
	Therefore $f*_\nu g\in \mathcal{B}(X^\prime,L^p_w(\nu))$ with $\|f*_\nu g\|_{\mathcal{B}(X^\prime,L^p_w(\nu))}\leq\|f\|_{\nu,p}\|g\|_\nu.$
	
	Now, let $g\in L^1(\nu).$ Then there exists two sequences ($\phi_n$) and ($\psi_n$) of simple functions converging to $f$ in $L^p(\nu)$ and $g$ in $L^1(\nu)$ respectively. For each $n\in\mathbb{N},$ it is clear that $\phi_n\in L^\infty(G)$ and $\psi_nh_{x^\prime}\in L^1(G)$. Thus, it follows that $\phi_n*(\psi_nh_{x^\prime})$ is bounded and hence  $\phi_n*_\nu\psi_n(x^\prime)=\phi_n*(\psi_nh_{x^\prime})\in L^p(\nu).$ Further, 
	\begin{align*}
	\|\phi_n*_\nu\psi_n(x^\prime)-f*_\nu g(x^\prime)\|_{\nu,p}\leq&\|(\phi_n-f)*_\nu\psi_n(x^\prime)\|_{\nu,p}+\|f*_\nu(\psi_n-g)(x^\prime)\|_{\nu,p}\\\leq&\|\phi_n-f\|_{\nu,p}\|\psi_n\|_\nu+\|f\|_{\nu,p}\|\psi_n-g\|_\nu.
	\end{align*}
	Thus the sequence ($\phi_n*_\nu\psi_n(x^\prime)$) converges to $f*_\nu g(x^\prime)$ in $\|\cdot\|_{\nu,p}$ norm. Since $L^p(\nu)$ is a closed subspace of $L^p_w(\nu),$ it follows that $f*_\nu g\in \mathcal{B}(X^\prime, L^p(\nu)).$
\end{proof}
Here is an analogue of the Young's inequality for the convolution of functions with respect to a vector measure.
\begin{cor}\label{c1}Let $1\leq p<\infty$ and let $\nu\in M_{ac}(G,X)$ be a semivariation translation and inversion invariant vector measure with $\nu(G)\neq 0$. If $f\in L^p(\nu)$ and $g\in L^q_w(\nu)$ with $1\leq q\leq p^\prime,$ then $f*_\nu g\in \mathcal{B}(X^\prime,L^r_w(\nu))$ with $\|f*_\nu g\|_{\mathcal{B}(X^\prime,L^r_w(\nu))}\leq\|f\|_{\nu,p}\|g\|_{\nu,q}$ where $\frac{1}{p}+\frac{1}{q}=1+\frac{1}{r}.$ In particular, if $g\in L^q(\nu),$ then $f*_\nu g\in \mathcal{B}(X^\prime,L^r(\nu)).$
\end{cor}
\begin{proof}
	Let $x^\prime\in B_{X^\prime}$ and $f\in L^p(\nu).$ Let $T_{f,x^\prime}$ denote a linear operator on some function space given by $T_{f,x^\prime}(g)=f*_\nu g(x^\prime).$ Then by Theorem \ref{T1}, $T_{f,x^\prime}$ is a bounded linear operator from $L^1_w(\nu)$ to $L^p_w(\nu)$ with $\|T_{f,x^\prime}\|_{\mathcal{B}(L^1_w(\nu),L^p_w(\nu))}\leq\|f\|_{\nu,p}.$ Now let $g\in L^{p^\prime}_w(\nu).$ Then, by the H\"older's inequality and Proposition \ref{R1} we have,
	\begin{align*}
	|f*_\nu g(x^\prime)(t)|=&|f*(gh_{x^\prime})|=\left|\int_Gf(ts^{-1})g(s)\,d\langle\nu,x^\prime\rangle\right|\\=&\int_G|\tau_t\tilde{f}(s)||g(s)|\,d|\langle\nu,x^\prime\rangle|\leq\|\tau_t\tilde{f}\|_{\nu,p}\|g\|_{\nu,p^\prime}=\|f\|_{\nu,p}\|g\|_{\nu,p^\prime}.
	\end{align*}
	It follows that $T_{f,x^\prime}$ is also a bounded linear operator from $L^{p^\prime}(\nu)$ to $L^\infty_w(\nu)$ with $\|T_{f,x^\prime}\|_{\mathcal{B}(L^{p^\prime}(\nu),L^\infty_w(\nu))}\leq\|f\|_{\nu,p}.$ Hence the proof follows by an application of the interpolation theorem \cite[Theorem 3.4]{FMNP}. By Theorem \ref{T1}, the second part also follows similarly as above.
\end{proof}
\begin{prop}\label{A6}
	Let $\nu\in M_{ac}(G,X)$ be a semivariation translation invariant vector measure with $\nu(G)\neq 0$. If $f,g\in L^1(\nu)$ then for $\phi\in L^\infty(G)\mbox{ and }x^\prime\in X^\prime,$
	$$\langle f*_\nu g(x^\prime),\phi\rangle=\left\langle\int_G(\tilde{f}*\phi)g\,d\nu,x^\prime\right\rangle.$$
\end{prop}
\begin{proof}
	Let $f\in L^1(\nu).$ Then by Theorem \ref{A3}, $f\in L^1(G).$ For $\phi\in L^\infty(G)$ we have $\tilde{f}*\phi\in L^\infty(G).$ Therefore $(\tilde{f}*\phi)g\in L^1(\nu).$ Then the proof of this follows from the Fubini's theorem.
\end{proof}
Now we find the Fourier transform of the convolution.
\begin{thm}
	Let $\nu\in M_{ac}(G,X)$ be a semivariation translation invariant vector measure with $\nu(G)\neq 0$ If $f,g\in L^1(\nu),$ then for $[\pi]\in\widehat{G}$ and $x^\prime\in X^\prime,$
	$$\widehat{f*_\nu g(x^\prime)}(\pi)=d_\pi\langle\langle\widehat{g}^\nu(\pi)\widehat{f}(\pi),x^\prime\rangle\rangle.$$
	Further, for $1\leq i,j\leq d_\pi,$ the $(i,j)^{\mbox{th}}$-entry of $\widehat{g}^\nu(\pi)\widehat{f}(\pi)$ is given by  $$(\widehat{g}^\nu(\pi)\widehat{f}(\pi))_{ij}=\frac{1}{d_\pi^2}\int_G\tilde{f}*\overline{\pi(t)_{ji}}g(t)\,d\nu(t).$$
\end{thm}
\begin{proof}
	Let $f,g\in L^1(\nu).$ Then for $x^\prime\in X^\prime$ and $[\pi]\in\widehat{G},$ we have, 
	\begin{align*}
	\widehat{f*_\nu g(x^\prime)}(\pi)_{ij}=&\widehat{f*gh_{x^\prime}}(\pi)_{ij}=(d_\pi\widehat{gh_{x^\prime}}(\pi)\widehat{f}(\pi))_{ij}=d_\pi\sum_{k=1}^{d_\pi}\widehat{gh_{x^\prime}}(\pi)_{ik}\widehat{f}(\pi)_{kj}\\=&d_\pi\sum_{k=1}^{d_\pi}\langle\widehat{g}^\nu(\pi)_{ik}
	,x^\prime\rangle\widehat{f}(\pi)_{kj}
	=d_\pi\sum_{k=1}^{d_\pi}\langle\widehat{g}^\nu(\pi)_{ik}\widehat{f}(\pi)_{kj},x^\prime\rangle\\=&d_\pi\left\langle\sum_{k=1}^{d_\pi}\widehat{g}
	^\nu(\pi)_{ik}\widehat{f}(\pi)_{kj},x^\prime\right\rangle=d_\pi\langle(\widehat{g}^\nu(\pi)\widehat{f}(\pi))_{ij},x^\prime\rangle.
	\end{align*}
	Hence $\widehat{f*_\nu g(x^\prime)}(\pi)=[d_\pi\langle(\widehat{g}^\nu(\pi)\widehat{f}(\pi))_{ij},x^\prime\rangle]=d_\pi\langle\langle\widehat{g}^\nu(\pi)\widehat{f}(\pi),x^\prime\rangle\rangle.$ By taking $\phi(\cdot)=\overline{\pi(\cdot)_{ji}}$ in Proposition \ref{A6} we have,
	\begin{align*}
	\left\langle\int_G(\tilde{f}*\overline{\pi(\cdot)_{ji}})g\,d\nu,x^\prime
	\right\rangle=&\langle f*_\nu g(x^\prime),\overline{\pi(\cdot)_{ji}}\rangle=\int_Gf*_\nu g(x^\prime)(t)\overline{\pi(t)_{ji}}\,dm_G(t)\\=&d_\pi\widehat{f*_\nu g(x^\prime)}(\pi)_{ij}=\langle d_\pi^2(\widehat{g}^\nu(\pi)\widehat{f}(\pi))_{ij},x^\prime\rangle.\qedhere
	\end{align*}
\end{proof}
Now we define the vector valued convolution.
\begin{defn}The vector valued convolution, with respect to $\nu,$ of two measurable functions $f$ and $g,$ denoted $f*^\nu g,$ is defined as $$f*^\nu g(t)=\int_Gf(ts^{-1})g(s)d\nu(s),$$ provided that the mapping $s\mapsto f(ts^{-1})g(s)$ belongs to $L^1(\nu)$ for $m_G$-almost everywhere $t\in G.$
\end{defn}
\begin{rem}\label{E1}
	Let $\nu\in M_{ac}(G,X).$ If $f\in L^p(G),\ 1\leq p<\infty$ and $g\in L^1_w(\nu)$ are such that the mapping $s\mapsto f(ts^{-1})g(s)\in L^1(\nu)$ for $m_G$-almost everywhere $t\in G$ then for $x^\prime\in X^\prime,$ we have $f*_\nu g(x^\prime)=\langle f*^\nu g,x^\prime\rangle.$
\end{rem}
Before, we proceed to the main results, here are some definitions.

Let $1\leq p<\infty.$ A function $f:G\rightarrow X$ is said to be Dunford $p$-integrable (for $p=1$ we say Dunford integrable) if $\langle f,x^\prime\rangle\in L^p(G),~x^\prime\in X^\prime.$ We denote by $L^p_w(G,X)$ the space of Dunford $p$-integrable functions equipped with the norm $$\|f\|_{L^p_w(G,X)}=\underset{x^\prime\in B_{X^\prime}}{\sup}\|\langle f,x^\prime\rangle\|_p.$$ A Dunford $p$-integrable function $f$ is said to be Pettis $p$-integrable (for $p=1$ we say Pettis integrable) if for each $A\in\mathfrak{B}(G)$ there exists a unique $x_A\in X$ such that $\int_A\langle f,x^\prime\rangle\,dm_G=\langle x_A,x^\prime\rangle,~x^\prime\in X^\prime.$ The vector $x_A$ is denoted by $(P)\int_Af\,dm_G.$ For more information on Dunford and Pettis integrability, see \cite{DU, T}.

\begin{thm}\label{T2}
	Let $\nu\in M_{ac}(G,X).$ If $f\in L^1(G)$ and $g\in L^1_w(\nu)$ such that the mapping $s\mapsto f(ts^{-1})g(s)$ is in $L^1(\nu)$ for $m_G$-almost everywhere $t\in G,$ then $f*^\nu g$ is Dunford integrable with $\|f*^\nu g\|_{L^1_w(G,X)}\leq\|f\|_1\|g\|_\nu.$ In particular, if $g\in L^1(\nu),$ then $f*^\nu g$ is Pettis integrable with $$(P)\int_Gf*^\nu g\,dm_G=\int_Gf\,dm_G\int_Gg\,d\nu.$$
\end{thm}
\begin{proof}
	Let $x^\prime\in X^\prime.$ Note that, by Radon-Nikodym theorem, there exists $h_{x^\prime}\in L^1(G)$ such that $d\langle\nu,x^\prime\rangle=h_{x^\prime}dm_G.$ Then by Remark \ref{E1}, $\langle f*^\nu g(t),x^\prime\rangle=f*gh_{x^\prime}(t),$ $m_G$-a.e. $t\in G$. Therefore the mapping $t\mapsto\langle f*^\nu g(t),x^\prime \rangle$ is measurable. Further, by Lemma \ref{bddconvo}, we have, $\|\langle f*^\nu g,x^\prime \rangle\|_1\leq\|f\|_1\|g\|_\nu\|x^\prime\|_{X^\prime}.$ Thus $f*^\nu g$ is Dunford integrable and $\|f*^\nu g\|_{L^1_w(G,X)}\leq\|f\|_1\|g\|_\nu.$
	
	We now prove the second statement. Let $g\in L^1(\nu).$ Note that for any $f\in L^1(G)$ and $A\in\mathfrak{B}(G),$ the mapping $s\mapsto\int_A f(ts^{-1})\ dm_G(t)$ is a bounded measurable function and therefore $$\int_G\int_A f(ts^{-1})\ dm_G(t)g(s)\,d\nu(s)\in X.$$ Let $x_A=\int_G\int_A f(ts^{-1})\,dm_G(t)g(s)\,d\nu(s).$ It follows, by an application of the Fubini's theorem, that $$\int_A\langle f*^\nu g,x^\prime \rangle\,dm_G=\langle x_A,x^\prime \rangle.$$ Thus $f*^\nu g$ is Pettis integrable and 
	\begin{equation*}
	(P)\int_Gf*^\nu g\,dm_G=x_G=\int_Gf\,dm_G\int_Gg\,d\nu.\qedhere
	\end{equation*}
\end{proof}
\begin{prop}\label{c2}Let $1\leq p<\infty$ and let $\nu\in M_{ac}(G,X)$ be a semivariation translation invariant vector measure with $\nu(G)\neq 0$. If $f\in L^p(\nu)$ and $g\in L^1_w(\nu)$ are such that the mapping $s\mapsto f(ts^{-1})g(s)$ belongs to $L^1(\nu)$ for $m_G$-almost everywhere $t\in G,$ then $f*^\nu g$ is Dunford $p$-integrable with $\|f*^\nu g\|_{L^p_w(G,X)}\leq \|f\|_{\nu,p}\|g\|_\nu\|\nu(G)\|^{-1/p}.$ In particular, if $g\in L^1(\nu),$ then $f*^\nu g$ is Pettis $p$-integrable.
\end{prop}
\begin{proof}
	Let $x^\prime\in B_{X^\prime}.$ By Remark \ref{E1} and Theorem \ref{T1}, $\|\langle f*^\nu g,x^\prime\rangle\|_{\nu,p}=\|f*_\nu g(x^\prime)\|_{\nu,p}\leq\|f\|_{\nu,p}\|g\|_\nu.$ Then, by Theorem \ref{A3}, $f*^\nu g$ is Dunford $p$-integrable with $\|f*^\nu g\|_{L^p_w(G,X)}\leq \|f\|_{\nu,p}\|g\|_\nu\|\nu(G)\|^{-1/p}.$ 
	
	For $g\in L^1(\nu),$ as in the proof of Theorem \ref{T2}, we have, for each $A\in \mathfrak{B}(G)$ there exists $x_A=\int_G\int_A f(ts^{-1})\,dm_G(t)g(s)\,d\nu(s)$ such that $\int_A\langle f*^\nu g,x^\prime \rangle\,dm_G=\langle x_A,x^\prime \rangle.$ Thus $f*^\nu g$ is Pettis $p$-integrable.
\end{proof}
Here is an analogue of the Young's inequality for the vector-valued convolution.
\begin{prop}\label{c3}Let $1\leq p<\infty$ and let $\nu\in M_{ac}(G,X)$ be a semivariation translation and inversion invariant vector measure with $\nu(G)\neq 0$. If $f\in L^p(\nu)$ and $g\in L^q_w(\nu),\ 1\leq q\leq p^\prime$ satisfying that  $s\mapsto f(ts^{-1})g(s)\in L^1(\nu)$ for $m_G$-almost everywhere $t\in G,$ then $f*^\nu g$ is Dunford $r$-integrable with $\|f*^\nu g\|_{L^r_w(G,X)}\leq\|f\|_{\nu,p}\|g\|_{\nu,q}\|\nu(G)\|^{-1/r}$ where $\frac{1}{p}+\frac{1}{q}=1+\frac{1}{r}.$ In particular, if $g\in L^q(\nu),$ then $f*^\nu g$ is Pettis $r$-integrable.
\end{prop}
\begin{proof}
The proof of this follows exactly as in the previous Proposition, except that one will have to use Corollary \ref{c1} instead of Theorem \ref{T1}.
\end{proof}

\section{Fourier transform of vector measures}
In this section, we define the Fourier transform of a vector measure. We show that the Fourier transform, considered as an operator, is completely bounded. We find a sufficient condition on the space $X$ so that the Fourier transform of a vector measure satisfies the Riemann-Lebesgue Lemma.
\begin{defn}
The Fourier transform of a vector measure $\nu$ at $[\pi]\in\widehat{G}$ is defined by $$\widehat{\nu}(\pi)=\frac{1}{d_\pi}\int_G \pi(t)^*d\nu(t)\in\mathcal{B}(\mathcal{H}_\pi)\otimes_{min}X.$$
\end{defn}
\begin{rem}
As mentioned in Remark \ref{RforFTM}, the Fourier transform of a vector measure $\nu$ at $[\pi]\in\widehat{G}$ can be identified as a $d_\pi\times d_\pi$ matrix and for $1\leq i,j\leq d_\pi,$ the $(i,j)^{\mbox{th}}$-entry of $\widehat{\nu}(\pi)$ is given by $\frac{1}{d_\pi}\int_G\overline{\pi(t)_{ji}}\ d\nu(t).$
\end{rem}

The following proposition shows that the Fourier transform of $f\in L^1(\nu)$ and the corresponding measure $\nu_f$ coincide.
\begin{prop}\label{FTVMC}
If $f\in L^1(\nu),$ then $\widehat{f}^\nu=\widehat{\nu_f}.$ If $f\in L^1_w(\nu)$ and $\nu\in M_{ac}(G,X),$ then for $x^\prime\in X^\prime$ and $[\pi]\in\widehat{G},$ $\widehat{f}_\nu(x^\prime)(\pi)=\langle\langle\widehat{\nu_f}(\pi),x^\prime\rangle\rangle.$
\end{prop}
\begin{proof}
If $f\in L^1(\nu)$ then $\widehat{\nu_f}=\widehat{f}^\nu$ follows from the definition of $\nu_f.$ If $f\in L^1_w(\nu)$ and $\nu\in M_{ac}(G,X),$ then for $x^\prime\in X^\prime$ and $[\pi]\in\widehat{G},$ $$\widehat{fh_{x^\prime}}(\pi)_{ij}=\int_Gf(t)h_{x^\prime}(t)\overline{\pi(t)_{ji}}\,dm_G(t)=\int_Gf(t)\overline{\pi(t)_{ji}}\,d\langle\nu,x^\prime\rangle(t)=\langle\widehat{\nu_f}(\pi)_{ij},x^\prime\rangle,$$
where $h_{x^\prime}=\frac{d\langle\nu,x^\prime\rangle}{dm_G}.$ Hence we have 
\begin{align*}
\widehat{f}_\nu(x^\prime)(\pi)&=\widehat{fh_{x^\prime}}(\pi)=[\widehat{fh_{x^\prime}}(\pi)_{ij}]=[\langle\widehat{\nu_f}(\pi)_{ij},x^\prime\rangle]=\langle\langle\widehat{\nu_f}(\pi),x^\prime\rangle\rangle. \qedhere
\end{align*}
\end{proof}
The following theorem shows the complete boundedness of the Fourier transform operator.
\begin{thm}\mbox{ }
\begin{enumerate}
\item [(i)] If $\nu\in M(G,X)$ then $\widehat{\nu}\in\ell^\infty\mbox{-}\underset{[\pi]\in\widehat{G}}{\oplus}M_{d_\pi}(X).$ In fact, $\underset{[\pi]\in\widehat{G}}{\sup}\|\widehat{\nu}(\pi)\|_{d_\pi}\leq\|\nu\|.$
\item [(ii)] The Fourier transform operator $\mathcal{F}$ from $M(G,X)$ to $\ell^\infty\mbox{-}\underset{[\pi]\in\widehat{G}}{\oplus}M_{d_\pi}(X)$ given by $\mathcal{F}(\nu)=\widehat{\nu}$, is completely bounded.
\end{enumerate}
\end{thm}
\begin{proof}
The proof of this is exactly same as in the proof of Theorem \ref{CBFTO}, once we use Theorem \ref{measure} and hence we omit it.
\end{proof}
\begin{prop}[Uniqueness theorem]\label{p3}
Let $\nu\in M_{ac}(G,X)$. If $\widehat{\nu}=0$ then $\nu=0.$
\end{prop}
\begin{proof}
Let $x^\prime\in X^\prime$ and $[\pi]\in\widehat{G}.$ First, we claim that if $\nu\in M_{ac}(G,X)$ then $\langle\widehat{\nu}(\pi),x^\prime\rangle=\widehat{h_{x^\prime}}(\pi),$ where $h_{x^\prime}$ is the Radon-Nikodym derivative $\frac{d\langle\nu,x^\prime\rangle}{dm_G}.$ For $1\leq i,j\leq d_\pi,$ we have, $$\langle\widehat{\nu}(\pi)_{ij},x^\prime\rangle=\frac{1}{d_\pi}\int_G\overline{\pi(t)_{ji}}\,d\langle\nu,x^\prime\rangle(t)=\frac{1}{d_\pi}\int_G\overline{\pi(t)_{ji}}h_{x^\prime}(t)\,dm_G(t)=\widehat{h_{x^\prime}}(\pi)_{ij}$$ and hence the claim.

Now, by our assumption that $\widehat{\nu}=0,$ it follows that $\widehat{h_{x^\prime}}=0$ for every $x^\prime\in X^\prime.$ Since $h_{x^\prime}\in L^1(G),$ by the classical Uniqueness theorem for Fourier transform, we have $h_{x^\prime}=0$ $m_G$-a.e. for every $x^\prime\in X^\prime.$ Let $A\in\mathfrak{B}(G)$ and $x^\prime\in X^\prime.$ Then, $$|\langle\nu,x^\prime\rangle|(A)=\int_A\,d|\langle\nu,x^\prime\rangle|=\int_A|h_{x^\prime}|\,dm_G=0.$$ 
This implies that each Borel subset of $G$ is $\nu$-null. Hence the proof.
\end{proof}

A natural question that arises here is whether the vector measure satisfies the Riemann-Lebesgue Lemma. In general a vector measure $\nu$ does not satisfy it, i.e., $\widehat{\nu}\notin c_0\mbox{-}\underset{[\pi]\in\widehat{G}}{\oplus}M_{d_\pi}(X).$ It is not true even for a compact abelian group. Here we provide an example using the Example \ref{EFRLL}.
\begin{exam}
Let the vector measure $\nu$ and $f$ be as given in the Example \ref{EFRLL}. By Proposition \ref{FTVMC},
$\widehat{\nu_f}=\widehat{f}^\nu.$ Thus by Example \ref{EFRLL}, we
have $\widehat{\nu_f}\notin
c_0\mbox{-}\underset{[\pi]\in\widehat{G}}{\oplus}M_{d_\pi}(L^1(G)).$
\end{exam}
\begin{rem}\label{FTVOp}
For any vector measure $\nu$ we have $\widehat{\nu}(\pi)=[T_\nu(\overline{\pi(\cdot)_{ji}})]_{d_\pi\times d_\pi}.$
\end{rem}
If $\nu\in M_{ac}(G,X)$ then also the Riemann-Lebesgue Lemma may not hold in general. An example is provided below.
\begin{exam}
Let $G$ be an infinite compact group, $X=L^1(G)$ and $\nu(A)=\chi_A,\ A\in\mathfrak{B}(G)$. Then $\nu\in M_{ac}(G,L^1(G))$ and moreover, it is clear that $T_\nu$ is just the  inclusion map from $C(G)$ to $L^1(G).$ Further, it follows from Remark \ref{FTVOp}, that, for $\ [\pi]\in\widehat{G},$ $\widehat{\nu}(\pi)=\pi(\cdot)^\ast.$ Thus $\|\widehat{\nu}(\pi)\|_{M_{d_\pi}(L^1(G))}=1.$ Hence $\nu$ does not satisfy the Riemann-Lebesgue Lemma.
\end{exam}
Our next result gives a sufficient condition for a vector measure to satisfy the Riemann-Lebesgue Lemma.
\begin{prop}\label{FTVMRNP}
Let $\nu\in M_{ac}(G,X)$ be a measure of bounded variation. If X has the Radon-Nikodym Property with respect to $(G,\mathfrak{B}(G),m_G),$ then $\widehat{\nu}\in c_0\mbox{-}\underset{[\pi]\in\widehat{G}}{\oplus}M_{d_\pi}(X).$
\end{prop}
\begin{proof}
Let $\nu\in M_{ac}(G,X)$ be a measure of bounded variation. Then by the definition of the Radon-Nikodym Property of $X$  there exists $f\in L^1(G,X)$ such that $d\nu=f\,dm_G.$ It is clear that $\widehat{\nu}=\widehat{f},$ where $\widehat{f}$ denotes the Fourier transform of the $X$-valued function $f.$ See \cite{GP}. Hence, by \cite[Corollary 3.8]{GP}, $\widehat{\nu}\in c_0\mbox{-}\underset{[\pi]\in\widehat{G}}{\oplus}M_{d_\pi}(X)$.
\end{proof}

\section{Convolution of a vector measure and a scalar measure}
In this section, we define the convolution of a vector measure and a scalar measure and study its properties.

For $\mu\in M(G)$ and $A\in\mathfrak{B}(G),$ note that the mapping $t\mapsto\mu(At^{-1})$ is  bounded and the following definition is well-defined. 
\begin{defn}
The convolution of $\mu\in M(G)$ and $\nu\in M(G,X)$ is defined by $$\mu*\nu(A)=\int_G\mu(At^{-1})\,d\nu(t),\ A\in\mathfrak{B}(G).$$
\end{defn}
Note that $\mu*\nu\in M(G,X)$ with $\|\mu*\nu\|\leq\|\mu\|\|\nu\|.$
\begin{prop}\label{p2}
If $\mu\in M(G)$ and $\nu\in\mathcal{M}(G,X),$ then $\mu*\nu\in\mathcal{M}(G,X)$. Also $T_{\mu*\nu}=T_\nu\circ C_\mu$ where $C_\mu$ is a mapping on $C(G)$ given by $C_\mu(\phi)(s)=\int_G\phi(ts)\,d\mu(t).$
\end{prop}
As the proof of this is similar to \cite[Lemma 4.3]{B}, we shall omit it. Now we find the Fourier transform of the convolution.
\begin{thm}
If $\mu\in M(G)$ and $\nu\in\mathcal{M}(G,X)$ then for $[\pi]\in\widehat{G},$ $$\widehat{\mu*\nu}(\pi)=d_\pi\widehat{\nu}(\pi)\widehat{\mu}(\pi).$$
\end{thm}
\begin{proof}Let $[\pi]\in\widehat{G}.$ Note that for $1\leq i,j\leq d_\pi$ and $t,s\in G,$ $$\overline{\pi(ts)_{ji}}=\sum_{k=1}^{d_\pi}\overline{\pi_{jk}(t)\pi_{ki}(s)}.$$ Using Proposition \ref{p2}, 
\begin{align*}
\widehat{\mu*\nu}(\pi)_{ij}=&\frac{1}{d_\pi}\int_G\overline{\pi(t)_{ji}}\,d\mu*\nu(t)=\frac{1}{d_\pi}T_\nu(C_\mu(\overline{\pi(\cdot)_{ji}}))=\frac{1}{d_\pi}T_\nu\left(\int_G\overline{\pi(t\cdot)_{ji}}\,d\mu(t)\right)\\=&\frac{1}{d_\pi}T_\nu\left(\int_G\sum_{k=1}^{d_\pi}\overline{\pi_{jk}(t)\pi_{ki}(\cdot)}\,d\mu(t)\right)=\sum_{k=1}^{d_\pi}\widehat{\mu}(\pi)_{kj}T_\nu(\overline{\pi_{ki}(\cdot)})\\=&d_\pi\sum_{k=1}^{d_\pi}\widehat{\mu}(\pi)_{kj}\widehat{\nu}(\pi)_{ik}=d_\pi\sum_{k=1}^{d_\pi}\widehat{\nu}(\pi)_{ik}\widehat{\mu}(\pi)_{kj}=d_\pi(\widehat{\nu}(\pi)\widehat{\mu}(\pi))_{ij},
\end{align*}
where $(\widehat{\nu}(\pi)\widehat{\mu}(\pi))_{ij}$ is the $(i,j)$th entry of the matrix $\widehat{\nu}(\pi)\widehat{\mu}(\pi)\in M_{d_\pi}(X).$ Hence 
\begin{align*}
\widehat{\mu*\nu}(\pi)&=[(\widehat{\mu*\nu}(\pi))_{ij}]_{d_\pi\times d_\pi}=[d_\pi(\widehat{\nu}(\pi)\widehat{\mu}(\pi))_{ij}]_{d_\pi\times d_\pi}=d_\pi\widehat{\nu}(\pi)\widehat{\mu}(\pi). \qedhere
\end{align*}
\end{proof}
Now we define another convolution of measures.
\begin{defn}\label{d3}The convolution of $\nu\in M(G,X)$ and $\mu\in M(G)$ is defined by $$\nu*\mu(A)=\int_G\nu(At^{-1})\,d\mu(t),~A\in\mathfrak{B}(G),$$ provided that $t\mapsto\nu(At^{-1})\in L^1(\mu).$ 
\end{defn}
For $\nu\in\mathcal{M}(G,X)$ and $A\in\mathfrak{B}(G),$ note that the mapping $t\mapsto\nu(At^{-1})$ is measurable and bounded by $\|\nu\|$. Therefore $t\mapsto\nu(At^{-1})\in L^1(\mu).$ Hence the Definition \ref{d3} makes sense. Also note that $\nu*\mu\in M(G,X)$ with $\|\nu*\mu\|\leq\|\nu\|\|\mu\|.$
\begin{prop}
The group $G$ is abelian if and only if $\mu*\nu=\nu*\mu$ for all $\mu\in M(G)$ and $\nu\in\mathcal{M}(G,X).$ 
\end{prop}
\begin{proof}
If $G$ is abelian then $\mu*\nu=\nu*\mu$, see \cite[Proposition 4.5]{B}. Conversely, suppose that $\mu*\nu=\nu*\mu$ for all $\mu\in M(G)$ and $\nu\in\mathcal{M}(G,X).$ Let $x^\prime\in X^\prime$ and $A\in\mathfrak{B}(G).$ Then, $\langle\mu*\nu(A),x^\prime\rangle=\langle\nu*\mu(A),x^\prime\rangle,$ 
i.e., $\mu*\langle\nu,x^\prime\rangle(A)=\langle\nu,x^\prime\rangle*\mu(A).$
Let $t,s\in G,$ and $x_0\in X.$ Now choose $x_0^\prime\in X^\prime$ such that $\langle x_0,x_0^\prime\rangle=1.$ Consider $\mu=\delta_t$, the Dirac measure on G at $t$ and choose $\nu=x_0\delta_s\in\mathcal{M}(G,X).$ Note that $\langle\nu,x_0^\prime\rangle=\delta_s.$ Then we obtain $\delta_t*\delta_s=\delta_s*\delta_t,$ i.e., $\delta_{ts}=\delta_{st}.$ Hence $ts=st$ for every $t,s\in G.$
\end{proof}

\section{Integrability properties of the convolution product}
In this section, we prove some integrability properties for the convolution product. Finally, we show that the usual convolution of a function in $L^1(\nu)$ with a function in $L^p(G)$ belongs to $L^p(\nu).$

For $f\in L^1(G)$ and $\nu\in M(G,X),$ we write $f*\nu=\mu_f*\nu$ where $\mu_f$ is given by $d\mu_f=f\,dm_G.$ We say that $f*\nu\in C(G,X)$ if $d(f*\nu)=f_\nu\,dm_G,$ for some $f_\nu\in C(G,X).$ Next result gives an analogue of the Young's inequality for $f*\nu$.
\begin{prop}\label{A1}\mbox{} Let $\nu\in M(G,X)$.
\begin{enumerate}[(i)]
\item Let $1\leq p<\infty.$ If $f\in L^p(G)$ then $f*\nu\in P_p(G,X).$ Further $$\|f*\nu\|_{P_p(G,X)}\leq\|f\|_p\|\nu\|.$$
\item Let $1<p<\infty.$ If $\nu\in M_p(G,X)$ and $f\in L^q(G)$ with $q^\prime>p,$ then $f*\nu\in P_r(G,X)$ where $\frac{1}{p}+\frac{1}{q}=1+\frac{1}{r}.$ Further, $$\|f*\nu\|_{P_r(G,X)}\leq\|f\|_q\|\nu\|_{p,m_G}.$$
\end{enumerate}
\end{prop}
\begin{proof}
(i) By density, it is enough to prove for continuous functions. Let $f\in C(G).$ Then $f*\nu\in C(G,X)$ and hence, by definition, $d(f*\nu)=f_\nu dm_G,$ for some $f_\nu\in C(G,X).$ By following the lines as in \cite[Proposition 4.7]{B}, it can be shown that $f_\nu(\cdot)=\int_G f(\cdot s^{-1})d\nu(s).$ Let $x^\prime\in X^\prime.$ Then, by the H\"older's inequality and the Fubini's theorem, it follows that $\int_G|\langle f_\nu(t),x^\prime \rangle|^p\,dm_G(t)\leq\|f\|_p^p\|\nu\|^p\|x^\prime\|_{X^\prime}^p.$ Thus (i) follows.

(ii) Let $x^\prime\in X^\prime$ and $\nu\in M_p(G,X).$ By Lemma \ref{Mp}, there exists $h_{x^\prime}\in L^p(G)$ such that $d\langle\nu,x^\prime\rangle=h_{x^\prime}\,dm_G.$ Let $f\in C(G).$ As mentioned in (i), $d(f*\nu)=f_\nu dm_G,$ where $f_\nu(\cdot)=\int_G f(\cdot s^{-1})\,d\nu(s).$ Then $\int_G|\langle f_\nu(t),x^\prime \rangle|^r\,dm_G(t)\leq\int_G(|f|*|h_{x^\prime}|(t))^r\,dm_G(t).$ Thus by the classical Young's inequality, the needed inequality follows. As $C(G)$ is dense in $L^q(G),$ the proof is complete.
\end{proof}
\begin{prop}\label{A2}\mbox{}\begin{enumerate}[(i)]
\item If $\nu\in\mathcal{M}(G,X),$ then $\nu_f\in\mathcal{M}(G,X)$ for every $f\in L^1(\nu).$
\item Let $1<p\leq\infty.$ If $\nu\in M_p(G,X)$ and $f\in L^q(G)$ for some $p^\prime\leq q\leq\infty,$ then $\nu_f\in M_r(G,X)$ where $\frac{1}{p}+\frac{1}{q}=\frac{1}{r}.$ Further, $\|\nu_f\|_{r,m_G}\leq\|\nu\|_{p,m_G}\|f\|_q.$
\end{enumerate}
\end{prop}
\begin{proof}
(i) By Lemma \ref{dense}, we can assume that $f\in C(G).$ In order to do this, by \cite[Corollary 14, Pg. 159]{DU}, it is enough to show that the operator $T_{\nu_f}$ is weakly compact. Now, note that $T_{\nu_f} = T_\nu\circ M_f,$ where $M_f$ denotes the multiplication operator on $C(G),$ given by $M_f(g)=fg.$ The proof of (i) is complete, by \cite[Proposition 5.2, Pg. 183]{C} and the fact that $\nu$ is regular.

(ii) Let $f\in L^q(G).$ By Lemma \ref{Mp}, we have, $L^q(G)\subset L^{p^\prime}(G)\subset L^1(\nu)$ and hence $f\in L^1(\nu).$ Thus $\nu_f$ is well-defined. Further, note that $T_{\nu_f}=T_\nu\circ M_f,$ where $M_f$ denotes the multiplication operator from $L^p(G)$ to $L^r(G)$ given as in (i). Therefore, the adjoint $T_{\nu_f}^*\in \mathcal{B}(X^\prime, L^r(G))$ and hence by \cite[Theorem 1, Pg. 259]{D}, $\nu_f\in M_r(G,X).$ Moreover, $\|\nu_f\|_{r,m_G}=\|T_{\nu_f}\|\leq\|T_\nu\|\|M_f\|=\|\nu\|_{p,m_G}\|f\|_q.$
\end{proof}
\begin{thm}\label{t2}Let $\nu\in M(G,X).$
\mbox{}\begin{enumerate}[(i)]
\item Let $1\leq p<\infty.$ If $f\in L^p(G)$ and $g\in L^1(\nu),$ then $f*^\nu g\in P_p(G,X).$ Further, $$\|f*^\nu g\|_{P_p(G,X)}\leq\|f\|_p\|g\|_\nu.$$
\item If $\nu\in M_{p_1}(G,X), g\in L^{p_2}(G)$ and $f\in L^{p_3}(G)$ where $0<\frac{1}{p_1}+\frac{1}{p_2}<1$ and $\frac{1}{p_1}+\frac{1}{p_2}+\frac{1}{p_3}>1,$ 
then $f*^\nu g\in P_r(G,X)$ where $\frac{1}{p_1}+\frac{1}{p_2}+\frac{1}{p_3}=1+\frac{1}{r}$. Further, $$\|f*^\nu g\|_{P_r(G,X)}\leq\|\nu\|_{p_1,m_G}\|g\|_{p_2}\|f\|_{p_3}.$$
\end{enumerate}
\end{thm}
\begin{proof}Note that $f*^\nu g=f*\nu_g.$ Since $g\in L^1(\nu),$ it is clear that $\nu_g\in M(G,X)$ with $\|\nu_g\|=\|g\|_\nu.$ Thus (i) follows from Proposition \ref{A1}(i). 

We shall now prove (ii). By assumption $\frac{1}{p_1}+\frac{1}{p_2}<1$ and therefore $p_1^\prime<p_2$ and $p_1>1.$ Further, since $g\in L^{p_2}(G),$ by proposition \ref{A2}(ii), it follows that $\nu_g\in M_s(G,X)$ for some $s$ such that $\frac{1}{p_1}+\frac{1}{p_2}=\frac{1}{s}.$ Since, by assumption, $\frac{1}{s}+\frac{1}{p_3}>1$ and $0<\frac{1}{p_1}+\frac{1}{p_2}<1,$ it follows that $p_3^\prime>s$ and $1<s<\infty.$ Thus, by Proposition \ref{A1}(ii), $f*\nu_g\in P_r(G,X)$ for some $r$ such that $\frac{1}{s}+\frac{1}{p_3}=1+\frac{1}{r}$ that is $\frac{1}{p_1}+\frac{1}{p_2}+\frac{1}{p_3}=1+\frac{1}{r}.$ Further, by Proposition \ref{A1}(ii), we have $\|f*\nu_g\|_{P_r(G,X)}\leq\|f\|_{p_3}\|\nu_g\|_{s,m_G}.$ Now, by Proposition \ref{A2}(ii), $\|\nu_g\|_{s,m_G}\leq\|\nu\|_{p_1,m_G}\|g\|_{p_2}.$ Hence (ii).
\end{proof}

By Theorem \ref{A3}, it is clear that we can consider the classical convolution of functions from $L^p(\nu),$ $1\leq p<\infty$ and $L^1(G).$ Our next result is in this direction. This theorem is the vector measure analogue of \cite[Proposition 2.39]{F}. For the case of compact abelian groups see \cite[Theorem 6.3]{B}.
\begin{thm}\label{T4}Let $\nu\in\mathcal{M}(G,X)$ be a semivariation translation invariant vector measure with $\nu(G)\neq 0$ and let $1\leq p<\infty.$
\begin{enumerate}[(i)]
\item If $f\in L^1(\nu)$ and $g\in L^p(G),$ then $f*g\in L^p(\nu).$ Further, $$\|f*g\|_{\nu,p}\leq\|f\|_\nu\|g\|_p\|\nu(G)\|^{-1/{p^\prime}}.$$ 
\item If $f\in L^p(\nu)$ and $g\in L^1(G),$ then $f*g\in L^p(\nu).$ Further, $$\|f*g\|_{\nu,p}\leq\|f\|_{\nu,p}\|g\|_1.$$
\end{enumerate}
\end{thm}
\begin{proof}
Since $C(G)$ is dense in both $L^p(G)$ and $L^p(\nu)$ for all $1\leq p<\infty,$ it is enough to verify both (i) and (ii) for continuous functions. So, let $f,g\in C(G).$

We shall first assume that $p=1.$ Note that (i) and (ii) are same. By Proposition \ref{R1}, we have, for $x^\prime\in B_{X^\prime},$
\begin{align*}
\int_G|f*g(t)|\,d|\langle\nu,x^\prime\rangle|(t)\leq&\int_G\|\tau_sf\|_{L^1(|\langle\nu,x^\prime\rangle|)}|g(s)|\,dm_G(s)\\\leq&\int_G\underset{x^\prime\in B_{X^\prime}}{\sup}\|\tau_sf\|_{L^1(|\langle\nu,x^\prime\rangle|)}|g(s)|\,dm_G(s)\\=&\int_G\|\tau_sf\|_\nu|g(s)|\,dm_G(s)\\=& \int_G\|f\|_\nu|g(s)|\,dm_G(s)= \|f\|_\nu\|g\|_1.
\end{align*}
Thus we are done with the case when $p=1.$

Now, we shall assume that $1<p<\infty.$ We first prove (ii). By the H\"older's inequality we have,
\begin{align*}
|f*g(t)|=&\left|\int_G\tau_sf(t)g(s)\,dm_G(s)\right|\leq\int_G|\tau_sf(t)||g(s)|^{1/p}|g(s)|^{1/{p^\prime}}\,dm_G(s)\\\leq&
\left(\int_G|\tau_sf(t)|^p|g(s)|\,dm_G(s)\right)^{1/p}\left(\int_G|g(s)|\,dm_G(s)\right)^{1/{p^\prime}}\\=&\|g\|_1^{1/{p^\prime}}(|f|^p*|g|)^{1/p}(t)
\end{align*}
Thus, by the case $p=1,$ \begin{align*} 
\|f*g\|_{\nu,p}^p=&\|(f*g)^p\|_\nu\leq\|g\|_1^{p-1}\||f|^p*|g|\|_\nu\\\leq&\|g\|_1^{p-1}(\|f^p\|_\nu\|g\|_1)=\|f\|_{\nu,p}^p\|g\|_1^p.
\end{align*}
Hence (ii) follows. Now for (i), by the H\"older's inequality and by Theorem \ref{A3} we have
\begin{align*}
|f*g(t)|=&\left|\int_G\tau_sf(t)g(s)\,dm_G(s)\right|\leq\int_G|\tau_sf(t)|^{1/{p^\prime}}|\tau_sf(t)|^{1/p}|g(s)|\,dm_G(s)\\\leq&
\left(\int_G|\tau_sf(t)|\,dm_G(s)\right)^{1/{p^\prime}}\left(\int_G|\tau_sf(t)||g(s)|^p\,dm_G(s)\right)^{1/p}\\=&\|f\|_1^{1/{p^\prime}}(|f|*|g|^p)^{1/p}(t)\leq\|\nu(G)\|^{-1/{p^\prime}}\|f\|_\nu^{1/{p^\prime}}(|f|*|g|^p)^{1/p}(t)
\end{align*}
Now the remaining proof follows as done for (ii).
\end{proof}

\section*{Acknowledgement}
The first author would like to thank University Grants Commission, India, for providing the research grant.


\begin{thebibliography}{aaaa}
\bibitem{B} \textsc{O. Blasco}, {\it Fourier analysis for vector-measures on compact abelian groups},  Rev. R. Acad. Cienc. Exactas Fís. Nat. Ser. A Math. RACSAM. 110 (2016) 519-539.
\bibitem{CFNP} \textsc{J. M. Calabuig}, \textsc{F. Galaz-Fontes}, \textsc{E. M. Navarrete} and \textsc{E. A. S\'anchez-P\'erez}, {\it Fourier transforms and convolutions on $L^p$ of a vector measure on a compact Haussdorff abelian group}, J. Fourier Anal. Appl. 19 (2013) 312-332.
\bibitem{CSb} \textsc{N. D. Chakraborty} and \textsc{Santwana Basu}, {\it Spaces of $p$-tensor integrable functions and related Banach space properties}, Real Anal. Exchange, 34 (2008), 87-104.
\bibitem{C} \textsc{J. B. Conway}, {\it A course in functional analysis}, Second Edition, Springer, 1997.
\bibitem{D} \textsc{N. Dinculeanu}, {\it Vector measures}, VEB Deutscher Verlag der Wissenschaften, Berlin (1966).
\bibitem{DU} \textsc{J. Diestel} and \textsc{J. J. Uhl}, {\it Vector Measures}, Math. Surveys, vol. 15. Amer. Math. Soc., Providence (1977).
\bibitem{ER} \textsc{E. G. Effros} and \textsc{Z-J. Ruan}, {\it Operator spaces}, Oxford University Press, 2000.
\bibitem{F} \textsc{G. B. Folland}, {\it A Course in Abstract Harmonic Analysis}, CRC Press, Boca Raton, 1995.
\bibitem{FMNP} \textsc{A. Fern\'andez}, \textsc{F. Mayoral}, \textsc{F. Naranjo} and \textsc{E. A. S\'anchez-P\'erez}, {\it Complex interpolation of spaces of integrable functions with respect to a vector measure}, Collect. Math. 61(3) (2010), 241-252.
\bibitem{GP} \textsc{J. Garc\'{i}a-Cuerva} and \textsc{J. Parcet}, {\it Vector-valued Hausdorff-Young inequality on compact groups}, {\it Proc. London Math. Soc.}, 88(3) (2004) 796-816.
\bibitem{HR} \textsc{E. Hewitt} and \textsc{K.A. Ross}, {\it Abstract Harmonic Analysis II}, vol. 152, Springer, Berlin, 1970.
\bibitem{ORP}\textsc{S. Okada}, \textsc{W. J. Ricker}, and \textsc{E. A. S\'anchez P\'erez}, {\it Optimal Domain and Integral Extension of Operators acting in Function Spaces}, Oper. Theory Adv. Appl., vol. 180. Birkh\"auser, Basel (2008).
\bibitem{P} \textsc{G. Pisier}, {\it Introduction to the Theory of Operator Spaces}, London Mathematical Society
Lecture Notes Series, Vol. 294, Cambridge University Press, 2003.
\bibitem{S} \textsc{G. F. Stef\'ansson}, {\it Integration in vector spaces}, Illinois J. Math., 45 (2001), 925-938.
\bibitem{T} \textsc{M. Talagrand}, {\it Pettis Integral and Measure Theory}, Memoirs of the AMS, vol. 307, 1984.

\end{thebibliography}
\end{document}